\documentclass[a4paper,12pt]{article}

\usepackage{amsthm}
\usepackage{amsmath}
\usepackage{amssymb}
\usepackage{amsfonts}
\usepackage{latexsym}
\usepackage{mathrsfs}
\usepackage{mathtools} 

\usepackage[T1]{fontenc}
\usepackage{lmodern} 
\usepackage{bm} 

\numberwithin{equation}{section}

\allowdisplaybreaks[1]

\usepackage{enumitem}

\renewcommand{\Re}{\operatorname{Re}}

\DeclareMathOperator{\sgn}{sgn}

\newcommand{\ci}{\mathrm{i}}
\newcommand{\ce}{\mathrm{e}}
\newcommand{\cd}{\mathrm{d}}

\DeclarePairedDelimiter{\abs}{\lvert}{\rvert} 
\DeclarePairedDelimiter{\norm}{\lVert}{\rVert} 
\DeclarePairedDelimiter{\rbra}{(}{)} 
\DeclarePairedDelimiter{\cbra}{\{}{\}} 
\DeclarePairedDelimiter{\sbra}{[}{]} 

\newcommand{\bP}{\ensuremath{\mathbb{P}}}

\newcommand{\bR}{\ensuremath{\mathbb{R}}}

\newcommand{\cB}{\ensuremath{\mathcal{B}}}

\newcommand{\cD}{\ensuremath{\mathcal{D}}}

\newcommand{\cF}{\ensuremath{\mathcal{F}}}

\newcommand{\cH}{\ensuremath{\mathcal{H}}}

\theoremstyle{plain}
\newtheorem{Thm}{Theorem}[section]
\newtheorem{Lem}[Thm]{Lemma}
\newtheorem{Prop}[Thm]{Proposition}
\newtheorem{Cor}[Thm]{Corollary}

\theoremstyle{definition}

\theoremstyle{remark}

\theoremstyle{plain}
\newtheorem*{Thm*}{Theorem}
\newtheorem*{Lem*}{Lemma}
\newtheorem*{Prop*}{Proposition}
\newtheorem*{Cor*}{Corollary}
\newtheorem*{Conj*}{Conjecture}
\theoremstyle{definition}
\newtheorem*{Ass*}{Assumption}
\newtheorem*{Def*}{Definition}
\newtheorem*{Rem*}{Remark}
\theoremstyle{remark}
\newtheorem*{Eg*}{Example}

\makeatletter
\renewcommand\section{\@startsection {section}{1}{\z@}%
                                   {-3.5ex \@plus -1ex \@minus -.2ex}%
                                   {2.3ex \@plus.2ex}%
                                   {\normalfont\large\bf}}
\makeatother

\makeatletter
\renewcommand\subsection{\@startsection {subsection}{1}{\z@}%
                                   {-3.5ex \@plus -1ex \@minus -.2ex}%
                                   {2.3ex \@plus.2ex}%
                                   {\normalfont\normalsize\bf}}
\makeatother

\makeatletter
\renewcommand\subsubsection{\@startsection {subsubsection}{1}{\z@}%
                                   {-3.5ex \@plus -1ex \@minus -.2ex}%
                                   {2.3ex \@plus.2ex}%
                                   {\normalfont\normalsize\it}}
\makeatother

\usepackage[%
bookmarks=true,%
bookmarksdepth=3,%
bookmarksnumbered=true,%
setpagesize=false,%
pdftitle={Sample path behaviors of L\'{e}vy processes %
conditioned to avoid zero},%
pdfauthor={Shosei Takeda},%
pdfkeywords={one-dimensional L\'{e}vy process; %
conditioning; sample path properties; limit theorem}%
]{hyperref}

\title{\Large\textbf{Sample path behaviors of L\'{e}vy processes
conditioned to avoid zero}}
\author{Shosei Takeda\footnote{Rakunan High School,
Kyoto,
Japan. (email:
\texttt{takeda.shosei@gmail.com})}}
\date{}
\begin{document}
\maketitle
\begin{abstract}
  Takeda--Yano~\cite{me} determined
  the limit of L\'{e}vy processes conditioned to avoid zero via various random clocks
  in terms of Doob's \(h\)-transform, where
  the limit processes may differ according to the choice of random clocks.
  The purpose of this paper is to investigate sample path behaviors
  of the limit processes in long time and in short time.
\end{abstract}
{\small Keywords and phrases: one-dimensional L\'{e}vy process;
conditioning; sample path property; limit theorem \\
MSC 2020 subject classifications: 60G17 (60F05; 60G51; 60J25)}
\section{Introduction}
For a measure \(\mu\) and for a non-negative measurable
or an integrable function \(f\),
we write \(\mu\sbra{f}\) for the integral \(\int f\,\cd\mu\) for simplicity.

Let \(\cbra{\bP_x^{\mathrm{Bes}\, +}\colon x\ge 0}\)
(resp. \(\cbra{\bP_x^{\mathrm{Bes}\,-}\colon x\le 0}\))
denote the law of (resp.\ the negative of) the three-dimensional
Bessel process, starting from \(x\). For \(x\in\bR\),
let \(\bP_x^{\mathrm{sBes}}\) denote
the law of the symmetrized three-dimensional Bessel process
starting from \(x\), i.e., it holds that
\(\bP_x^{\mathrm{sBes}}=\bP_x^{\mathrm{Bes}\, +}\) for \(x>0\),
\(\bP_x^{\mathrm{sBes}}=\bP_x^{\mathrm{Bes}\, -}\) for \(x<0\)
and \(\bP_0^{\mathrm{sBes}}=\frac{1}{2}\bP_0^{\mathrm{Bes}\, +}
+\frac{1}{2}\bP_0^{\mathrm{Bes}\, -}\).
Let \(x\in \bR \) and
let \((X=(X_t,t\ge 0),\bP_x^B)\)
be the canonical representation of a standard Brownian motion starting from \(x\)
and \((\cF_t)\) denote the right-continuous filtration
generated by the natural filtration.
We denote by \(T_0=\inf\cbra{t>0\colon X_t=0}\) the first hitting time of the origin.
Then, in the case \(x\ne 0\),
we have the following conditioning limit theorem:
for any bounded \(\cF_t\)-measurable functional \(F_t\),
it holds that
\begin{align}\label{eq:intro}
  \lim_{s\to\infty}\bP_x^B\sbra{F_t| T_0>s} = \bP_x^{\mathrm{sBes}}\sbra{F_t}.
\end{align}
In the case \(x=0\), we have a similar conditioning limit theorem
if we replace \(\bP_0^B\) with the law of the Brownian meander.
This means that the Brownian motion conditioned to avoid zero
up to time \(t\) converges in law to the symmetrized three-dimensional Bessel process.
The left-hand side of~\eqref{eq:intro} can be regarded as the
\textit{Brownian motion conditioned to avoid zero}.
We remark that the three-dimensional Bessel process is transient
and never hits the origin.
For \(x\in\bR\setminus\cbra{0}\),
the process \(\bP_x^{\mathrm{sBes}}\) can be written
via Doob's \(h\)-transform with respect to
the non-negative harmonic function \(h(x)=\abs{x}\)
of the killed Brownian motion \(\cbra{\bP_x^{B,0}\colon x\in \bR\setminus\cbra{0}}\)
as follows:
\begin{align}
  \bP_x^{\mathrm{sBes}}|_{\cF_t}
  = \frac{\abs{X_t}}{\abs{x}}\bP_x^{B,0}|_{\cF_t},\quad t>0.
\end{align}
Let \(n^B\) stand for the Brownian excursion measure.
Then \(\bP_0^{\mathrm{sBes}}\) can also be written as
\begin{align}
  \bP_0^{\mathrm{sBes}}|_{\cF_t}
  = \frac{\abs{X_t}}{n^B\sbra{\abs{X_t}}}\cdot n^B|_{\cF_t},
  \quad t>0.
\end{align}

These results for Brownian motions were generalized to
one-dimensional L\'{e}vy processes.
Yano~\cite{MR2603019} constructed and investigated
one-dimensional L\'{e}vy processes conditioned to avoid zero
under the conditions that the process is symmetric and has no Gaussian part.
He also investigated path behaviors of the process.
Yano~\cite{MR3072331} extended their results
to asymmetric L\'{e}vy processes.
He also showed the existence of a non-negative harmonic function
for asymmetric killed L\'{e}vy processes
under some technical conditions.
Pant\'{\i}~\cite{MR3689384} and Tsukada~\cite{MR3838874}
showed the existence of the harmonic function
under more general conditions
and Pant\'{\i}~\cite{MR3689384} investigated asymmetric
L\'{e}vy processes conditioned to avoid zero using
\(h\)-transform with respect to its harmonic function.
Recently, Takeda--Yano~\cite{me} obtained
a family of harmonic functions \(h^{(\gamma)}\) parametrized by \(-1\le\gamma\le 1\)
for the killed L\'{e}vy process
under more general conditions.
They also constructed the L\'{e}vy process conditioned to avoid zero,
using the \(h^{(\gamma)}\)-transform.

In this paper, we investigate the path behaviors of
L\'{e}vy processes conditioned to avoid zero which is
constructed in Takeda--Yano~\cite{me}.

\subsection{L\'{e}vy processes conditioned to avoid zero}
We shall recall the construction of L\'{e}vy processes
conditioned to avoid zero in Takeda--Yano~\cite{me}.
For more details of the notation of this section,
see Section~\ref{Sec:preliminaries}.
Let \((X=(X_t,t\ge 0),(\bP_x)_{x\in\bR})\) denote the canonical representation
of a one-dimensional L\'{e}vy process
and we write \(\bP=\bP_0\).
Throughout this paper, we always assume
the following condition~\ref{item:assumption}:
\begin{enumerate}[label=\textbf{(\Alph*)}]
  \item The process \((X,\bP)\) is recurrent and, for each \(q > 0\), it holds that
        \begin{align}
          \int_0^\infty
          \abs*{\frac{1}{q+\varPsi(\lambda)}}
          \, \cd \lambda < \infty,
        \end{align}
        where \(\varPsi(\lambda)\) denotes the characteristic exponent
        given by \(\bP\sbra{\ce^{\ci \lambda X_t}}
        =\ce^{-t\varPsi(\lambda)}\).\label{item:assumption}
\end{enumerate}
Let \(T_A=\inf\cbra{t>0\colon X_t\in A}\) stand for the hitting time of
a set \(A\subset\bR\) and
we simply write \(T_a\coloneqq T_{\cbra{a}}\)
for the hitting time of a point \(a\in\bR\).
The condition~\ref{item:assumption} implies that
\(\bP(T_0=0)=1\) and \((X,\bP)\) is not
compound Poisson. In addition, there exists
the \(q\)-resolvent density \(r_q\) for \(q>0\).
For \(x\in\bR\), we define
\(h_q(x)=r_q(0)-r_q(-x)\ge 0\) and
\begin{align}
  h(x)= \lim_{q\to 0+}h_q(x), \quad x\in \bR,
\end{align}
which is called the \emph{renormalized zero resolvent};
see (\ref{Lem-item:h-exist}) of Lemma~\ref{Lem:h}.
The function \(h\) is subadditive;
see (\ref{Lem-item:h-subadditive}) of Lemma~\ref{Lem:h}.
We denote the second
moment of \(X_1\) by
\begin{align}
  m^2=\bP\sbra{X_1^2}\in (0,\infty].
\end{align}
For \(-1\le \gamma\le 1\), define
\begin{align}
  h^{(\gamma)}(x)=h(x)+\frac{\gamma}{m^2}x,\quad x\in\bR.
\end{align}
If \(m^2=\infty\), the functions \(h^{(\gamma)}\)
coincide with \(h\) for all \(-1\le\gamma\le 1\).
The function \(h^{(\gamma)}\) is non-negative
(see~\eqref{eq:h-g-plus}) and subadditive.
Let \(\bP_x^0\) denote the law under \(\bP_x\) of the killed process
\begin{align}
  X^0_t = \begin{cases}
    X_t    & \text{if \(t< T_0\)},   \\
    \Delta & \text{if \(t\ge T_0\)},
  \end{cases}
\end{align}
where \(\Delta\) stands for a cemetary point.
Let \(n\) denote It\^{o}'s excursion measure
normalized by the equation
\begin{align}\label{eq:n-normalize}
  n\sbra{1-\ce^{-qT_0}}=\frac{1}{r_q(0)},\quad q>0;
\end{align}
see Section~\ref{Sec:preliminaries}.
The next lemma says \(h^{(\gamma)}\) is harmonic for the killed process.
\begin{Lem}[\cite{me}]\label{Lem:harmonic}
  Assume the condition~\ref{item:assumption} is satisfied.
  For \(-1\le \gamma\le 1\) and \(x\in\bR\), it holds that
  \begin{align}\label{eq:harmonic}
    \bP^0_x\sbra{h^{(\gamma)}(X_t)} = h^{(\gamma)}(x)
    \quad\text{and}\quad
    n\sbra{h^{(\gamma)}(X_t)} = 1,\quad t>0.
  \end{align}
  In particular, the process \((h^{(\gamma)}(X_t),t> 0)\) is
  a non-negative \(\bP^0_x\)-martingale.
\end{Lem}
The proof of Lemma~\ref{Lem:harmonic} can be found in~\cite[Theorem 8.1]{me}
and~\cite[(iii) of Theorem 2.2]{MR3689384}.

For \(-1\le\gamma\le 1\),
define \(\cH^{(\gamma)} = \cbra{x\in\bR\colon h^{(\gamma)}(x)>0 }\) and
\(\cH^{(\gamma)}_0 = \cH^{(\gamma)} \cup \cbra{0}\).
If \(m^2=\infty\), we have \(\cH^{(\gamma)}=\cH^{(0)}\).
If \(m^2<\infty\), we have \(\cH^{(1)}_0\cap \cH^{(-1)}_0
\subset \cH^{(\gamma)}_0=\bR\) for \(-1<\gamma<1\) by~\eqref{eq:h-g-plus}.
Adopting Doob's \(h\)-transform approach,
we construct the \(h^{(\gamma)}\)-transform by
\begin{align}\label{eq:def-h-trans}
  \bP_x^{(\gamma)}|_{\cF_t}
  = \begin{dcases}
    \frac{h^{(\gamma)}(X_t)}{h^{(\gamma)}(x)}
    \cdot \bP^0_x|_{\cF_t}             & \text{if } x\in\cH^{(\gamma)}, \\
    h^{(\gamma)}(X_t) \cdot n|_{\cF_t} & \text{if } x = 0.
  \end{dcases}
\end{align}
Note that, if \(m^2=\infty\), we have \(\bP_x^{(\gamma)}=\bP_x^{(0)}\)
for all \(-1\le \gamma\le 1\).
By Lemma~\ref{Lem:harmonic}, we see that
\(\bP_x^{(\gamma)}|_{\cF_t}\) is consistent in \(t>0\)
and thus \(\bP_x^{(\gamma)}\) is well-defined
and is a probability measure on \(\cF_\infty\);
for more details, see Yano~\cite[Theorem 9.1]{yano2021universality}.
We can see \(\bP^{(\gamma)}_x(T_{\bR\setminus\cH^{(\gamma)}}>t)=1\) for all \(t>0\)
and consequently it holds that
\(\bP^{(\gamma)}_x(T_{\bR\setminus\cH^{(\gamma)}}=\infty)=1\).
Hence the process \((X, \bP^{(\gamma)}_x)\) never hits zero.
We remark that, for \(x\in\cH^{(\gamma)}\), the measure
\(\bP^{(\gamma)}_x\) is absolutely continuous on \(\cF_t\)
with respect to \(\bP_x\), but is singular on \(\cF_\infty\) to \(\bP_x\)
since
\(\bP_x^{(\gamma)}(T_0=\infty) = \bP_x(T_0<\infty)=1\).

The next theorem shows that for \(x\in \cH^{(\gamma)}\),
the measure \(\bP_x^{(\gamma)}\) can be obtained
as the limit measure of the L\'{e}vy process conditioned to avoid zero
via a random clock, i.e., a certain parametrized family of random times, going
to infinity.
Let \(\bm{e}\) be an independent exponential time with mean
\(1\) and we write \(\bm{e}_q=\bm{e}/q\).

\begin{Thm}[\cite{me}]\label{Thm:limit-meas}
  Assume the condition~\ref{item:assumption} is satisfied.
  Let \(t>0\) and \(F_t\) be a bounded \(\cF_t\)-measurable functional.
  Then the following assertions hold:
  \begin{enumerate}
    \item
          \(\displaystyle
          \lim_{q\to 0+}
          \bP_x\sbra{F_t|T_0>\bm{e}_q}=\bP_x^{(0)}\sbra{F_t}
          \), for \(x\in \cH^{(0)}\);\label{Thm-item:limit-meas-exp}
    \item
          \(\displaystyle
          \lim_{a\to\pm\infty}
          \bP_x\sbra{F_t|T_0>T_a}=\bP_x^{(\pm 1)}\sbra{F_t}
          \), for \(x\in \cH^{(\pm 1)}\);
    \item
          \(\displaystyle
          \lim_{\substack{a\to\infty,\,b\to\infty,\\ \frac{b-a}{a+b}\to \gamma}}
          \bP_x\sbra{F_t|T_0>T_{\cbra{a,-b}}}=\bP_x^{(\gamma)}\sbra{F_t}
          \), for \(-1\le \gamma\le 1\)
          and \(x\in \cH^{(\gamma)}\).\label{Thm-item:limit-meas-twohitting}
  \end{enumerate}
\end{Thm}
The proof of Theorem~\ref{Thm:limit-meas} can be found in Corollary 8.2 of~\cite{me}.
We remark, however, that
there are computational errors in the claim (iii) of Corollary 8.2 of~\cite{me};
see~\cite[Section 6]{iba2024}.
Claim (\ref{Thm-item:limit-meas-exp}) of Theorem~\ref{Thm:limit-meas}
is also proved in Pant\'{\i}~\cite[Theorem 2.7]{MR3689384}.
If \(m^2<\infty\), then the limit measure differs
according to the random clock.

We remark that the limit
\( \lim_{s\to\infty}
\bP_x\sbra{F_t|T_0>s}\) via constant clock is determined in the symmetric stable case
(see Yano--Yano--Yor~\cite{MR2552915})
but the limit is an open problem in the general L\'{e}vy case.

The left-hand side
of each of~(\ref{Thm-item:limit-meas-exp})--(\ref{Thm-item:limit-meas-twohitting})
of Theorem~\ref{Thm:limit-meas} can be regarded as
\emph{L\'{e}vy processes conditioned to avoid zero}
although the resulting processes may differ according to the choice of the clocks.
We remark that the resulting processes are characterized via Doob's \(h\)-transform.
For related studies, see
Chaumont~\cite{MR1419491} and
Chaumont--Doney~\cite{MR2164035} for
L\'{e}vy processes conditioned to stay positive.
Yano--Yano~\cite{MR3444297} for diffusions conditioned to avoid zero.

Let \(\cD\) denote the space of c\`{a}dl\`{a}g paths
\(\omega\colon [0,\infty)\to \bR\cup \cbra{\Delta}\).
We denote by \(\theta\) the shift operator and by \(k\) the killing operator, i.e.,
we define, for \(\omega\in \cD\) and \(t\ge 0\), \(\theta_t\omega(s)=\omega(s+t),\;
s\ge 0\), and define
\begin{align}
  k_t \omega(s)=\begin{cases}
    \omega(s) & \text{if \(s<t\),}    \\
    \Delta    & \text{if \(s\ge t\).}
  \end{cases}
\end{align}
For \(s>0\), we denote by \(g_s=\sup\cbra{u\le s\colon X_u=0}\)
the last hitting time of the origin up to time \(s\).
Then we have, for \(\tau>0\),
\begin{align}
  k_{\tau-g_\tau}\circ \theta_{g_\tau} \omega(s)=
  \begin{cases}
    \omega(g_\tau +s) & \text{if \(0\le s<\tau-g_\tau\)}, \\
    \Delta            & \text{if \(s\ge \tau-g_{\tau}\)}.
  \end{cases}\end{align}
The next theorem shows that for \(x=0\),
the measure \(\bP_x^{(\gamma)}=\bP_0^{(\gamma)}\) can be obtained
as the limit, via a random clock, of a measure similar to
the L\'{e}vy meander.

\begin{Thm}\label{Thm:limit-P0}
  Assume the condition~\ref{item:assumption} is satisfied.
  Let \(t>0\) and \(F_t\) be a bounded \(\cF_t\)-measurable functional.
  Then the following assertions hold:
  \begin{enumerate}
    \item
          \(\displaystyle
          \lim_{q\to 0+}
          \bP_0\sbra{F_t\circ k_{\bm{e}_q-g_{\bm{e}_q}}\circ \theta_{g_{\bm{e}_q}}}
          =\bP_0^{(0)}\sbra{F_t}
          \);\label{Thm-item:limit-P0-exp}
    \item
          \(\displaystyle
          \lim_{a\to\pm\infty}
          \bP_0\sbra{F_t\circ k_{{T_a}-g_{{T_a}}}\circ \theta_{g_{{T_a}}}}
          =\bP_0^{(\pm 1)}\sbra{F_t}
          \);\label{Thm-item:limit-P0-hitting}
    \item
          \(\displaystyle
          \lim_{\substack{a\to\infty,\,b\to\infty,\\ \frac{b-a}{a+b}\to \gamma}}
          \bP_0\sbra{F_t\circ k_{{T_{\cbra{a,-b}}}-g_{T_{\cbra{a,-b}}}}
          \circ \theta_{g_{T_{\cbra{a,-b}}}}}
          =\bP_0^{(\gamma)}\sbra{F_t}
          \), for \(-1\le \gamma\le 1\).\label{Thm-item:limit-P0-twohitting}
  \end{enumerate}
\end{Thm}
The proof of Theorem~\ref{Thm:limit-P0} will be given in Section~\ref{Sec:pf-meander}.
Claim (\ref{Thm-item:limit-P0-exp}) of Theorem~\ref{Thm:limit-P0} is also proved
in Pant\'{\i}~\cite[Theorem 2.8]{MR3689384}.

\subsection{Main results}\label{Subsec:result}
Recall that we always assume the condition~\ref{item:assumption}.

\subsubsection{Long-time behaviors of the process \((X,\bP_x^{(\gamma)})\)}
The proofs of the following Theorems~\ref{Thm:p-gamma-abs-infty},~\ref{Thm:P-g-equi}
and~\ref{Thm:h-gamma-process-infty} will be given in Section~\ref{Sec:pf}.
\begin{Thm}\label{Thm:p-gamma-abs-infty}
  Let \(-1\le \gamma\le 1\) and \(x\in \cH_0^{(\gamma)}\). Then it holds that
  \begin{align}
    \bP_x^{(\gamma)}\rbra*{\lim_{t\to\infty} \abs{X_t}=\infty} =1.
  \end{align}
  Consequently, the process \((X, \bP_x^{(\gamma)})\) is transient.
\end{Thm}

We discuss the result
when \(m^2<\infty\).
In this case, recall that, by~\eqref{eq:h-g-plus}, we have \(\cH^{(\gamma)}_0=\bR\)
for \(-1<\gamma<1\).
\begin{Thm}\label{Thm:P-g-equi}
  Assume \(m^2<\infty\).
  Then, for \(x\in \bR\) and \(-1<\gamma<1\),
  the measure \(\bP_x^{(\gamma)}\) is
  equivalent to \(\bP_x^{(0)}\).
  Moreover, for \(x\in \cH_0^{(\pm 1)}\), the measure
  \(\bP_x^{(\pm 1)}\)
  is absolutely continuous
  with respect to \(\bP_x^{(0)}\).
\end{Thm}
We discuss long-time behaviors of the process \((X,\bP^{(\gamma)}_x)\)
in the case \(m^2<\infty\).
Define
\begin{align}
  \Omega^{+}_\infty   & = \cbra*{\lim_{t\to\infty}X_t=\infty},                             \\
  \Omega^{-}_\infty   & = \cbra*{\lim_{t\to\infty}X_t=-\infty},                            \\
  \Omega^{+,-}_\infty & = \cbra*{\limsup_{t\to\infty}X_t=-\liminf_{t\to\infty}X_t=\infty}.
\end{align}
Then the sets \( \Omega^{+}_\infty\), \(\Omega^{-}_\infty \)
and \(\Omega^{+,-}_\infty\) are mutually disjoint and
\(\cbra{\lim_{t\to\infty}\abs{X_t}=\infty}
\subset \Omega^{+}_\infty\cup\Omega^{-}_\infty \cup \Omega^{+,-}_\infty\).
Hence by Theorem~\ref{Thm:p-gamma-abs-infty}, it holds that
\begin{align}\label{eq:Omega-union}
  \bP^{(\gamma)}_x(\Omega^{+}_\infty\cup \Omega^{-}_\infty\cup \Omega^{+,-}_\infty)=1.
\end{align}
If \(m^2<\infty\), the process \((X,\bP_x^{(\gamma)})\) drifts to
either \(+\infty\) or \(-\infty\)
with a certain probability.
\begin{Thm}\label{Thm:h-gamma-process-infty}
  Assume \(m^2<\infty\).
  Then, for \(-1\le \gamma\le 1\), it holds that
  \begin{align}
    \bP_x^{(\gamma)}\rbra{\Omega^{\pm}_\infty}
    = \begin{dcases}
      \frac{(1\pm\gamma)}{2}\frac{h^{(\pm 1)}(x)}{h^{(\gamma)}(x)}
       & \text{if \(x\in \cH^{(\gamma)}\),} \\
      \frac{1\pm\gamma}{2}
       & \text{if \(x=0\).}
    \end{dcases}
  \end{align}
  Consequently, for \(x\in\cH^{(1)}_0\cap\cH^{(-1)}_0 \),
  it holds that
  \begin{gather}
    \bP_x^{(\gamma)}
    \rbra{\Omega^{+}_\infty \cup \Omega^{-}_\infty}
    =
    \bP_x^{(1)}\rbra{\Omega^{+}_\infty}
    = \bP_x^{(-1)}\rbra{\Omega^{-}_\infty}
    =1,
  \end{gather}
  which implies that \(\bP_x^{(1)}\) and \(\bP_x^{(-1)}\)
  are mutually singular on \(\cF_{\infty}\).
\end{Thm}
Note that,
if \(m^2=\infty\), the process \(X\) can be oscillating
under \(\bP_x^{(\gamma)}=\bP_x^{(0)}\);
see, e.g., Theorem~\ref{Thm:stable-oscillate}.

\subsubsection{Short-time behaviors of the process \((X,\bP_x^{(\gamma)})\)}
The proofs of the
following Theorems~\ref{Thm:h^S/x},~\ref{Thm:gaussian-entrance},~\ref{Thm:entrance-one}
and~\ref{Thm:feller}
will be given in Section~\ref{Sec:short-time}.

We first deal with differential property at \(0\) of \(h\), which
is used for the discussion of short-time behaviors.
Since \(h\) and \(h^{(\gamma)}\) are subadditive, it holds that
\begin{align}
  h'(0\pm)                   & \coloneqq
  \lim_{x\to 0\pm}\frac{h(x)}{x}
  =\pm\sup_{x> 0}\frac{h(\pm x)}{x}      \\
  {h^{(\gamma)\prime}}(0\pm) & \coloneqq
  \lim_{x\to 0\pm}\frac{h^{(\gamma)}(x)}{x}
  =\pm\sup_{x> 0}\frac{h^{(\gamma)}(\pm x)}{x} =
  h'(0\pm)\pm \frac{\gamma}{m^2}.
\end{align}
By~(\ref{Lem-item:h/x-infty})
of Lemma~\ref{Lem:h}, we have
\begin{align}
  \abs{h'(0\pm)}                 & =\pm h'(0\pm) \in \sbra*{\frac{1}{m^2},\infty}, \\
  \abs{h^{(\gamma)\prime}(0\pm)} & =\pm h^{(\gamma)\prime}(0\pm)
  \in\sbra*{\frac{1\pm\gamma}{m^2},\infty}.
\end{align}
\begin{Thm}\label{Thm:h^S/x}
  It holds that
  \begin{align}\label{eq:h^S'}
    {h}'(0+)+\abs{{h}'(0-)}
    =\lim_{x\to 0+}\frac{h(x)+h(-x)}{x}=\frac{2}{\sigma^2}\in (0,\infty].
  \end{align}
  Consequently, for \(-1\le \gamma\le 1\), it holds that
  \begin{align}
    {h^{(\gamma)\prime}}(0+)+\abs{{h^{(\gamma)\prime}}(0-)}
    =\lim_{x\to 0+}\frac{h^{(\gamma)}(x)+h^{(\gamma)}(-x)}{x}
    =\frac{2}{\sigma^2}\in (0,\infty].
  \end{align}
\end{Thm}
We remark that
Winkel~\cite[Lemma 1]{MR1894112} already showed that
\begin{align}
  \lim_{x\to 0+}\frac{h_q(x)+h_q(-x)}{x} = \frac{2}{\sigma^2}, \quad q > 0.
\end{align}
By Theorem~\ref{Thm:h^S/x}, we see that
\(\sigma^2>0\) implies \(\abs{{h^{\prime}}(0\pm)}\le 2/\sigma^2 <\infty\)
and that \(\sigma^2=0\) implies \({{h^{\prime}}(0+)}=\infty\)
or \(-{{h^{\prime}}(0-)}=\infty\).

Define
\begin{align}
  \Omega^{+}_0
   & = \cbra{\exists t_0>0 \text{ such that } 0<\forall t<t_0, \, X_t>0}, \\
  \Omega^{-}_0
   & = \cbra{\exists t_0>0 \text{ such that } 0<\forall t<t_0, \, X_t<0}, \\
  \Omega^{+,-}_0
   & = \cbra{\exists\cbra{t_n} \text{ with } t_n\to {0+} \text{ such that
  } \forall n,\,
  X_{t_n}X_{t_{n+1}}<0}.
\end{align}
Then \(\Omega^{+}_0\), \(\Omega^{-}_0\) and \(\Omega^{+,-}_0\) are
mutually disjoint and
we have \(\Omega^{+}_0\cup \Omega^{-}_0\cup \Omega^{+,-}_0 = \cD\).
\begin{Thm}\label{Thm:gaussian-entrance}
  Assume \(m^2<\infty\) and \(\sigma^2>0\).
  Then, for \(-1\le\gamma\le 1\), it holds that
  \(\bP_0^{(\gamma)}(\Omega^{+,-}_0)=0\),
  \begin{align}
    \bP_0^{(\gamma)}(\Omega^{+}_0) = \frac{\sigma^2}{2}{h^{(\gamma)\prime}}(0+)
    \quad\text{and}\quad
    \bP_0^{(\gamma)}(\Omega^{-}_0) = \frac{\sigma^2}{2}\abs{{h^{(\gamma)\prime}}(0-)}.
  \end{align}
\end{Thm}

\begin{Thm}\label{Thm:entrance-one}
  Assume \(m^2<\infty\).
  If
  \(h'(0+)=\infty\) and \(\abs{h'(0-)}<\infty\),
  then \(\bP_0^{(\gamma)}(\Omega^{+}_0)=1\) for \(-1\le\gamma\le 1\).
  If \(h'(0+)<\infty\) and \(\abs{h'(0-)}=\infty\),
  then \(\bP_0^{(\gamma)}(\Omega^{-}_0)=1\) for \(-1\le\gamma\le 1\).
\end{Thm}

In the case \(\abs{h^{\prime}(0\pm)} =\infty\),
we do not obtain general properties.
We obtain the oscillating short-time behavior
under some technical assumptions.
\begin{Thm}\label{Thm:feller}
  Let \(-1\le\gamma\le 1\). Assume the following four assertions hold:
  \begin{enumerate}
    \item \(\displaystyle
          \liminf_{x\to\infty}h^{(\gamma)}(x)>0\)
          and \(\displaystyle
          \liminf_{x\to-\infty}h^{(\gamma)}(x)>0\);\label{item:cond:h-gamma-infty}
    \item \(\displaystyle\lim_{x\to
            0}\frac{h(x)}{\abs{x}}=\infty\),
          i.e., \(h'(0+)=-h'(0-)=\infty\);\label{item:cond:h/x-infty}
    \item \(\displaystyle\lim_{x\to 0}\frac{h_q(x+y)-h_q(y)}{h(x)}=1_{\cbra{y=0}}\) for
          all \(q>0\);\label{item:cond:h_q/h}
    \item \(\displaystyle0<\liminf_{x\to 0}\frac{h(-x)}{h(x)}\le \limsup_{x\to
            0}\frac{h(-x)}{h(x)}<\infty\).\label{item:cond:h-/h}
  \end{enumerate}
  Then it holds that \(\bP_0^{(\gamma)}(\Omega^{+,-}_0)=1\).
\end{Thm}
Note that,
If \(m^2<\infty\), (\ref{Lem-item:h/x-infty}) of Lemma~\ref{Lem:h}
implies that the condition~(\ref{item:cond:h-gamma-infty}) of Theorem~\ref{Thm:feller}
always holds for \(-1<\gamma<1\).

\subsection{Examples}
Before proceeding with the proofs of the results, we introduce some examples.
\subsubsection{Brownian motions}
Assume \((X,\bP)\) is a standard Brownian motion.
Then \(\sigma^2=m^2=1\) and \(h(x)=\abs{x}\).
By Theorems~\ref{Thm:h-gamma-process-infty} and~\ref{Thm:gaussian-entrance},
it holds that, for \(-1\le \gamma\le 1\),
\begin{align}
  \bP_0^{(\gamma)}(\Omega_\infty^+) & =
  \bP_0^{(\gamma)}(\Omega_0^+)=\frac{1+\gamma}{2},                                       \\
  \bP_0^{(\gamma)}(\Omega_\infty^-) & = \bP_0^{(\gamma)}(\Omega_0^-)=\frac{1-\gamma}{2}.
\end{align}
Since the Brownian motion has no jumps, the process \((X,\bP_x^{(\gamma)})\)
also has no jumps. Thus the avoiding zero process \((X,\bP_x^{(\gamma)})\)
does not change the sign. In fact, we have
\begin{align}
  \bP_0^{(\gamma)}= \frac{1+\gamma}{2}\bP_0^{\mathrm{Bes}\, +}
  +\frac{1-\gamma}{2}\bP_0^{\mathrm{Bes}\, -}.
\end{align}
Moreover, by Theorem~\ref{Thm:gaussian-entrance}, it holds that,
for \(-1\le \gamma\le 1\),
\begin{align}
  \bP_x^{(\gamma)}(\Omega_\infty^+)=1_{\cbra{x>0}},
  \quad
  \bP_x^{(\gamma)}(\Omega_\infty^-)=1_{\cbra{x<0}},
  \quad x\ne 0.
\end{align}
In fact, we have \(\bP_x^{(\gamma)} = \bP_x^{\mathrm{Bes}\, +}\) for \(x>0\)
and \(\bP_x^{(\gamma)} = \bP_x^{\mathrm{Bes}\, -}\) for \(x<0\).

\subsubsection{Stable processes}
Assume \((X,\bP)\) is strictly stable of index \(1<\alpha<2\).
Then \(m^2=\infty\) and the L\'{e}vy measure \(\nu\) can be written as
\begin{align}
  \nu(\cd x)=\begin{cases}
    c_{+} x^{-1-\alpha}\, \cd x      & \text{for \(x\in (0,\infty)\)},   \\
    c_{-} \abs{x}^{-1-\alpha}\,\cd x & \text{for \(x\in (-\infty, 0)\)},
  \end{cases}
\end{align}
where \(c_+\) and \(c_-\) are non-negative constants such that
\(c_++c_->0\).
The characteristic exponent \(\varPsi\) can be expressed as
\begin{align}
  \varPsi(\lambda)= c\abs{\lambda}^\alpha
  \rbra*{1-\ci \beta\sgn(\lambda)\tan \frac{\alpha\pi}{2}},
\end{align}
where \(c=-(c_++c_-)\Gamma(-\alpha)\cos(\pi\alpha/2)\)
and \(\beta = (c_+-c_-)/(c_++c_-)\);
see, e.g.,~\cite[Section 1.2]{MR3155252}.
We write \(c'=-c\beta \tan (\alpha\pi/2)\).
Then as a special case of~\eqref{eq:h_q}, it holds that
\begin{align}
  h_q(x)=\frac{1}{\pi}\int_0^\infty
  \Re \rbra*{\frac{1-e^{\ci \lambda x}}{q+(c+\ci c')\abs{\lambda}^\alpha}}
  \,\cd \lambda.
\end{align}
The dominated convergence theorem
implies that
\begin{align}\label{eq:stable-h-inte}
  h(x)=\frac{1}{\pi}\int_0^\infty
  \Re \rbra*{\frac{1-e^{\ci \lambda x}}{(c+\ci c')\abs{\lambda}^\alpha}}
  \,\cd \lambda.
\end{align}
Thus we have
\begin{align}\label{eq:h-repre-stable}
  h(x)-h_q(x)
  = \frac{1}{\pi} \int_0^\infty
  \Re\rbra*{\frac{1-e^{\ci \lambda x}}{(c+\ci c')\abs{\lambda}^\alpha}
  \frac{q}{q+(c+\ci c')\abs{\lambda}^\alpha}}
  \, \cd \lambda.
\end{align}
Hence it is obvious that
\begin{align}
  h_q'(x)=h'(x)+v(x), \quad x\in\bR\setminus\cbra{0},
\end{align}
where \(v(x)\) is a bounded continuous function.
Yano~\cite{MR3072331} calculated~\eqref{eq:stable-h-inte}
and obtained
\begin{align}\label{eq:stable-h-Yano}
  h(x)=-\frac{\alpha\Gamma(-\alpha)\sin(\pi\alpha/2)(1-\beta\sgn(x))}
  {c(1+\beta^2\tan^2(\pi\alpha/2))}
  \abs{x}^{\alpha-1};
\end{align}
see also Pant\'{\i}~\cite[Example 5.1]{MR3689384}.

Assume \((X,\bP)\) is spectrally positive (resp.\ negative), i.e.,
\(\beta=1\) (resp. \(\beta=-1\)). Then by~\eqref{eq:stable-h-Yano},
we have \(\cH_0^{(0)}=(-\infty, 0]\)
(resp. \(\cH_0^{(0)}=[0,\infty)\)).
On the other hand, assume \((X,\bP)\) is not spectrally one-sided, i.e., \(-1<\beta<1\).
Then the functions \(h\) and \(h_q\) satisfy the assumptions
of Theorem~\ref{Thm:feller}. Hence we obtain the following theorem:

\begin{Thm}
  Assume \((X,\bP)\) is a strictly stable process
  of index \(1<\alpha<2\).
  If \((X,\bP)\) is spectrally positive (resp.\ negative), it holds that
  \begin{align}
    \bP^{(0)}_0(\Omega_0^{-})=1 \quad \text{(resp. \( \bP^{(0)}_0(\Omega_0^{+})=1 \)).}
  \end{align}
  If \((X,\bP)\) is not spectrally one-sided, it holds that
  \begin{align}
    \bP^{(0)}_0(\Omega_0^{+,-})=1.
  \end{align}
\end{Thm}

Furthermore,
we obtain the following long-time behavior:
\begin{Thm}\label{Thm:stable-oscillate}
  Assume \((X,\bP)\) is a strictly stable process
  of index \(1<\alpha<2\).
  If \((X,\bP)\) is spectrally positive (resp.\ negative), it holds that
  \begin{align}
    \bP^{(0)}_0(\Omega_\infty^{-})=1 \quad
    \text{(resp. \( \bP^{(0)}_0(\Omega_\infty^{+})=1 \)).}
  \end{align}
  If \((X,\bP)\) is not spectrally one-sided, it holds that
  \begin{align}\label{eq:stable-oscillate}
    \bP^{(0)}_0(\Omega_\infty^{+,-})=1.
  \end{align}
\end{Thm}
To prove~\eqref{eq:stable-oscillate}, we use the same discussion as
the proof of~\cite[Corollary 1.4]{MR2603019}.
\subsubsection{Recurrent spectrally negative processes}
Let \((X,\bP)\) be a spectrally negative L\'{e}vy process, i.e.,
\(\nu(0,\infty)=0\), satisfying the assumption~\ref{item:assumption}.
Then,~\cite[Example 5.2]{MR3689384} says that
\begin{align}
  h(x) = W(x)-\frac{x}{m^2},
\end{align}
where \(W(x)\) stands for the scale function of \(X\).
Since \(W(x)=0\) for \(x\le 0\), we have
\(h(x)=\abs{x}/{m^2}\in [0,\infty)\) for \(x\le 0\).
If \(m^2=\infty\), we have \(\cH_0^{(0)}=[0,\infty)\)
and hence it holds that
\begin{align}
  \bP^{(0)}_x(\Omega_\infty^{+})=\bP^{(0)}_0(\Omega_0^{+})=1,\quad
  x\in [0,\infty).
\end{align}

\subsubsection{Symmetric processes}
We consider the case \((X,\bP)\) is symmetric and satisfies
the condition~\ref{item:assumption}.
Then \(h(x)=h(-x)\).
If \(\sigma^2>0\), we have
\(h'(0+)=-h'(0-)=1/\sigma^2\)
and hence, by Theorem~\ref{Thm:gaussian-entrance},
\begin{align}
  \bP_0^{(\gamma)}(\Omega_0^{+})=\bP_0^{(\gamma)}(\Omega_0^{-})=\frac{1}{2},
  \quad -1\le \gamma\le 1.
\end{align}
On the other hand, we assume \(\sigma^2=0\)
and \(\lambda \mapsto \Re \varPsi(\lambda)\) is eventually non-decreasing.
Then, Theorem~\ref{Thm:h^S/x} implies that
the function \(h\) satisfies
the assumption
(\ref{item:cond:h/x-infty})
of Theorem~\ref{Thm:feller}.
Moreover,
Lemma 4.4 and (i) of Lemma 6.2 of Yano~\cite{MR2603019} states that
\(h\) and \(h_q\) satisfy the
assumption~(\ref{item:cond:h_q/h})
of Theorem~\ref{Thm:feller}.
The function \(h\) obviously satisfies the
assumption~(\ref{item:cond:h-/h})
of Theorem~\ref{Thm:feller}.
Hence, if~(\ref{item:cond:h-gamma-infty}) of Theorem~\ref{Thm:feller}
also holds, it holds that
\begin{align}
  \bP_0^{(\gamma)}(\Omega_0^{+,-})=1,\quad -1\le\gamma\le 1.
\end{align}

\subsection{Outline of the paper}

The paper is organized as follows.
In Section~\ref{Sec:preliminaries}, we prepare
general properties of L\'{e}vy processes and
some preliminary facts of the renormalized zero resolvent \(h\).
In Sections~\ref{Sec:pf} and~\ref{Sec:short-time},
we prove the main results for long-time behaviors and
short-time behaviors, respectively.
In Section~\ref{Sec:resolvent} as an appendix, we investigate
the resolvent density under \(\bP_x^{(\gamma)}\).
In Section~\ref{Sec:pf-meander} as another appendix,
we will give the proof of Theorem~\ref{Thm:limit-P0}.

\section{Preliminaries}\label{Sec:preliminaries}
\subsection{General properties of L\'{e}vy processes}

Let \((X=(X_t,t\ge 0),\bP_x)\) denote the canonical representation
of a one-dimensional L\'{e}vy process starting from \(x\in\bR\)
on the c\`{a}dl\`{a}g path space \(\cD\)
and we write \(\bP=\bP_0\).
For \(t\ge 0\), we denote by \(\cF_t^X = \sigma(X_s, 0\le s\le t)\) the natural
filtration of \(X\) and we write \(\cF_t = \bigcap_{s>t} \cF_s^X\)
and \(\cF_\infty = \sigma(\bigcup_{t>0}\cF_t)\).
It is well-known that we have
\begin{align}
  \bP\sbra{\ce^{\ci\lambda X_t}}=\ce^{-t\varPsi(\lambda)},
  \quad \text{for \(t\ge 0\) and \(\lambda\in\bR\),}
\end{align}
where \(\varPsi(\lambda)\) denotes the characteristic exponent
of \(X\) given by the L\'{e}vy--Khintchine formula
\begin{align}
  \varPsi(\lambda)
  = \ci v \lambda
  + \frac{1}{2} \sigma^2 \lambda^2
  + \int_\bR \rbra*{1 - \ce^{\ci \lambda x} + \ci \lambda x 1_{\cbra{\abs{x}< 1}}}
  \nu(\cd x)
\end{align}
for some constants \(v \in \bR\) and \(\sigma^2 \ge 0\)
and a characteristic measure \(\nu\) on \(\bR\)
which satisfies \(\nu(\cbra{0})=0\) and
\begin{align}
  \int_\bR \rbra*{x^2 \wedge 1} \nu(\cd x) < \infty.
\end{align}
The measure \(\nu\) is called
a L\'{e}vy measure.
See, e.g.,~\cite{MR1406564,MR3155252}.

We consider the following four conditions:
\begin{enumerate}[label=\textbf{(A\arabic*)}]
  \item The process \((X,\bP)\) is not a compound Poisson process;\label{item:cond-not-CPP}
  \item \(0\) is regular for itself, i.e., \(\bP(T_0=0)=1\);\label{item:cond-regular}
  \item \(\displaystyle \int_\bR \Re\rbra*{\frac{1}{q+\varPsi(\lambda)}}\,\cd\lambda<\infty\)
        for all \(q>0\);\label{item:cond-Reinte}
  \item We have either \(\sigma^2>0\) or
        \(\displaystyle \int_{(-1,1)}\abs{x}\nu(\cd x)=\infty\).\label{item:cond-nu}
\end{enumerate}
Then the following lemma is well-known:
\begin{Lem}\label{Lem:equi-four-assump}
  The following three assertions hold:
  \begin{enumerate}
    \item The conditions~\ref{item:cond-not-CPP} and~\ref{item:cond-regular} hold
          if and only if the conditions~\ref{item:cond-Reinte}
          and~\ref{item:cond-nu} hold;\label{Lem-item:equi-1234}
    \item Under the condition~\ref{item:cond-Reinte},
          the condition~\ref{item:cond-regular}
          holds if and only if the
          condition~\ref{item:cond-nu} holds;\label{Lem-item:equi-24}
    \item   The condition~\ref{item:cond-Reinte} holds if and only if
          \((X,\bP)\) has the bounded \(q\)-resolvent density \(r_q\),
          which satisfies
          \begin{align}
            \int_\bR f(x) r_q(x) \, \cd x
            = \bP\sbra*{\int_0^\infty \ce^{-qt}f(X_t) \, \cd t},\quad q>0,
          \end{align}
          for all non-negative measurable functions \(f\).
          Moreover, under the condition~\ref{item:cond-Reinte},
          the condition~\ref{item:cond-regular} holds if and only if
          \(x \mapsto r_q(x)\) is continuous.\label{Lem-item:equi-resol}
  \end{enumerate}
\end{Lem}
For the proofs of~(\ref{Lem-item:equi-1234}) and~(\ref{Lem-item:equi-24})
of Lemma~\ref{Lem:equi-four-assump}, see Kesten~\cite{MR0272059} and
Bretagnolle~\cite{MR0368175}.
For the proof of~(\ref{Lem-item:equi-resol}) of Lemma~\ref{Lem:equi-four-assump},
see Theorems II.16 and II.19 of Bertoin~\cite{MR1406564}.

Throughout this paper, we always assume the condition~\ref{item:assumption}.
This implies that \((X,\bP)\) has the bounded continuous resolvent density
which is given by
\begin{align}
  r_q(x)
  =\frac{1}{\pi}\int_0^\infty
  \Re\rbra*{\frac{\ce^{-\ci\lambda x}}{q+\varPsi(\lambda)}}
  \, \cd \lambda
\end{align}
for all \(q > 0\) and \(x\in \bR\);
see, e.g., Winkel~\cite[Lemma 2]{MR1894112}
and Tsukada~\cite[Colorally 15.1]{MR3838874}.
Combining this and Lemma~\ref{Lem:equi-four-assump},
we see the condition~\ref{item:assumption}
implies~\ref{item:cond-not-CPP}--\ref{item:cond-nu}.
Tsukada~\cite[Lemma 15.5]{MR3838874} also proved that
the condition~\ref{item:assumption} implies
\begin{align}\label{eq:inte-of-varPsi}
  \int_0^\infty \abs*{\frac{1\wedge \lambda^2}{\varPsi(\lambda)}}
  \, \cd \lambda < \infty;
\end{align}
see also~\cite[Lemma 2.4]{me}.

Under the condition~\ref{item:assumption}, we denote
by \(L=(L_t, t\ge 0)\) local time at \(0\) normalized by
the equation
\begin{align}\label{eq:regularity-of-L}
  \bP_x\sbra*{\int_0^\infty \ce^{-qt} \cd L_t} = r_q(-x),
  \quad x \in \bR;
\end{align}
see, e.g.,~\cite[Section V]{MR1406564}.
Let \(n\) denote the characteristic measure of
excursions away from \(0\), called It\^{o}'s excursion measure
(see, e.g.,~\cite[Section IV.4]{MR1406564}).
Then the equation~\eqref{eq:n-normalize} holds.

\subsection{The renormalized zero resolvent}
We define
\begin{align}\label{eq:h_q}
  h_q(x)=r_q(0)-r_q(-x)
  =\frac{1}{\pi}\int_0^\infty
  \Re\rbra*{\frac{1-\ce^{\ci \lambda x}}{q+\varPsi(\lambda)}} \, \cd \lambda.
\end{align}
Since we have
\begin{align}\label{eq:-qT_0}
  \bP_x\sbra{\ce^{-qT_0}}=\frac{r_q(-x)}{r_q(0)} \ge 0
\end{align}
(see, e.g., Bertoin~\cite[Colorally II.18]{MR1406564}),
the function \(h_q\) is non-negative.
In addition, \(h_q\) is subadditive, i.e., \(h_q(x+y)\le h_q(x)+h_q(y)\)
for \(x,y\in \bR\); see, e.g., the proof of
Lemma 3.3 in~\cite{MR3689384} and
the proofs of (ii) and (iii) of Theorem 1.1 in~\cite{me}.

We denote the second moment of \(X_1\) by
\begin{align}
  m^2 = \bP\sbra{X_1^2}=\sigma^2+\int_\bR x^2\nu(\cd x)\in (0,\infty].
\end{align}

\begin{Lem}[The renormalized zero resolvent]\label{Lem:h}
  Assume the condition~\ref{item:assumption} is satisfied.
  Then the following assertions hold:
  \begin{enumerate}
    \item for \(x\in\bR\), the limit \( h(x)\coloneqq \lim_{q\to 0+} h_q(x)\) exists and is finite,
          which is called the \emph{renormalized zero resolvent};\label{Lem-item:h-exist}
    \item \(h\) is non-negative, continuous and subadditive \((h(x+y)\le h(x)+h(y)\) for \(x,y\in\bR)\)
          and \(h(0)=0\);\label{Lem-item:h-subadditive}
    \item \(\displaystyle h(x)+h(-x)
          =\frac{2}{\pi}\int_0^\infty
          \Re\rbra*{\frac{1-\cos \lambda x}{\varPsi(\lambda)}} \, \cd \lambda\),
          for \(x\in\bR\);\label{Lem-item:h^S-repre}
    \item if \(m^2<\infty\), it holds that
          \(\displaystyle h(x)=\frac{1}{\pi}\int_0^\infty
          \Re\rbra*{\frac{1-\ce^{\ci \lambda x}}{\varPsi(\lambda)}} \, \cd \lambda\),
          for \(x\in\bR\);
    \item \(\displaystyle \lim_{x\to\infty}\frac{h(x)}{\abs{x}}=\frac{1}{m^2}
          \in [0,\infty)\);\label{Lem-item:h/x-infty}
    \item \(\displaystyle \lim_{y\to\pm\infty}
          \cbra{h(x+y)-h(y)}=\pm \frac{x}{m^2}\in\bR\).\label{Lem-item:h-diff-infty}
  \end{enumerate}
\end{Lem}
The proof of Lemma~\ref{Lem:h} can be found in
Theorems 1.1 and 1.2 and Lemma 3.3 of~\cite{me}.

We define, for \(-1\le \gamma\le 1\),
\begin{align}\label{eq:def-h-g}
  h^{(\gamma)}(x)=h(x)+\frac{\gamma}{m^2}x,\quad x\in\bR.
\end{align}
By Lemma~\ref{Lem:h}, the function
\(h^{(\gamma)}\) is subadditive, \(h^{(\gamma)}(0)=0\)
and
\begin{gather}\label{eq:h-g/x-infty}
  \lim_{x\to\pm\infty} \frac{h^{(\gamma)}(x)}{\abs{x}} = \frac{1\pm\gamma}{m^2}.
\end{gather}
By subadditivity of \(h^{(\gamma)}\) and by~\eqref{eq:h-g/x-infty}, we also have
\begin{align}\label{eq:h-g-plus}
  h^{(\gamma)}(\pm x)\ge \frac{1\pm \gamma}{m^2}{x} \ge 0, \quad \text{for all \(x\ge 0\)}.
\end{align}

For \(-1\le\gamma\le 1\),
we define \(\cH^{(\gamma)} = \cbra{x\in\bR\colon h^{(\gamma)}(x)>0 }\) and
\(\cH^{(\gamma)}_0 = \cH^{(\gamma)} \cup \cbra{0}\).
By recurrence of \(X\), continuity and subadditivity of \(h\) and~\eqref{eq:harmonic},
\(\cH^{(\gamma)}_0\) is either \(\bR\), \([0,\infty)\) or \((-\infty,0]\).
In addition, if \(m^2<\infty\) and \(-1<\gamma<1\), then~\eqref{eq:h-g-plus}
implies that \(\cH^{(\gamma)}_0=\bR\).

Then we can define the \(h^{(\gamma)}\)-transformed process
given by~\eqref{eq:def-h-trans}.


\section{The long-time behaviors}\label{Sec:pf}

We prepare some important \(\bP_x^{(\gamma)}\)-martingale
which is used for investigating the path behaviors of the process.
Recall that we always assume the condition~\ref{item:assumption}.
\begin{Lem}\label{Lem:P-g-mart}
  Let \(-1\le \gamma\le 1\) and \(x\in \cH_0^{(\gamma)}\).
  Then \((\frac{1}{h^{(\gamma)}(X_t)},t> 0)\)
  is a non-negative \(\bP_x^{(\gamma)}\)-supermartingale.
  Moreover,
  for \(\gamma_1, \gamma_2\in [-1,1]\),
  the process
  \((\frac{h^{(\gamma_2)}(X_t)}{h^{(\gamma_1)}(X_t)},t>0)\)
  is a non-negative \(\bP_x^{(\gamma_1)}\)-martingale,
  and its mean is \(\frac{h^{(\gamma_2)}(x)}{h^{(\gamma_1)}(x)}\)
  if \(x\in \cH^{(\gamma_1)}\)
  and is \(1\) if \(x=0\).
\end{Lem}
\begin{proof}
  Recall that we have \(P_x^{(\gamma)}(T_{\bR\setminus \cH^{(\gamma)}}=\infty)=1\)
  for \(x\in \cH^{(\gamma)}_0\); see just after~\eqref{eq:def-h-trans}.
  This implies \(h^{(\gamma)}(X_t)\ne 0\),
  \(\bP_x^{(\gamma)}\)-a.s.
  We first assume \(x\in \cH^{(\gamma)}\).
  Let \(0<s<t\) and let \(F_s\) be a non-negative bounded \(\cF_s\)-measurable functional.
  Then we have
  \begin{align}
    \bP_x^{(\gamma)}\sbra*{\frac{1}{h^{(\gamma)}(X_t)}F_s}
    = \frac{1}{h^{(\gamma)}(x)}\bP_x\sbra{F_s; T_0 > t}
    \le\frac{1}{h^{(\gamma)}(x)} \bP_x\sbra{F_s; T_0 > s}
    = \bP_x^{(\gamma)}\sbra*{\frac{1}{h^{(\gamma)}(X_s)}F_s},
  \end{align}
  which implies that \((\frac{1}{h^{(\gamma)}(X_t)}, t>0)\) is a
  non-negative
  \(\bP_x^{(\gamma)}\)-supermartingale.
  By Lemma~\ref{Lem:harmonic}, we have,
  for \(\gamma_1,\gamma_2\in [-1,1]\)
  and \(x\in \cH^{(\gamma_1)}\),
  \begin{align}\label{eq:h/h-gamma-mart}
    \bP_x^{(\gamma_1)}
    \sbra*{\frac{h^{(\gamma_2)}(X_t)}{h^{(\gamma_1)}(X_t)}F_s}
     & =\frac{1}{h^{(\gamma_1)}(x)}\bP_x\sbra*{h^{(\gamma_2)}(X_t)F_s;T_0>t} \\
     & =\frac{1}{h^{(\gamma_1)}(x)}\bP_x\sbra*{h^{(\gamma_2)}(X_s)F_s;T_0>s} \\
     & =  \bP_x^{(\gamma_1)}
    \sbra*{\frac{h^{(\gamma_2)}(X_s)}{h^{(\gamma_1)}(X_s)}F_s},
  \end{align}
  which implies that \((\frac{h^{(\gamma_2)}(X_t)}{h^{(\gamma_1)}(X_t)}, t> 0)\) is
  a non-negative \(\bP_x^{(\gamma_1)}\)-martingale with mean
  \(\frac{h^{(\gamma_2)}(x)}{h^{(\gamma_1)}(x)}\).

  We next assume \(x=0\). Then we have
  \begin{align}
    \bP_0^{(\gamma)}\sbra*{\frac{1}{h^{(\gamma)}(X_t)}F_s}
    = n\sbra{F_s; T_0> t }
    \le n\sbra{F_s; T_0 > s}
    = \bP_0^{(\gamma)}\sbra*{\frac{1}{h^{(\gamma)}(X_s)}F_s},
  \end{align}
  which implies that \((\frac{1}{h^{(\gamma)}(X_t)}, t>0)\) is a
  non-negative
  \(\bP_0^{(\gamma)}\)-supermartingale.
  By Lemma~\ref{Lem:harmonic}, we have
  \begin{align}\label{eq:h/h-gamma-mart-n}
    \bP_0^{(\gamma_1)}
    \sbra*{\frac{h^{(\gamma_2)}(X_t)}{h^{(\gamma_1)}(X_t)}F_s}
     & =n\sbra*{h^{(\gamma_2)}(X_t)F_s;T_0>t}
    = n\sbra*{h^{(\gamma_2)}(X_s)F_s;T_0>s}
    =  \bP_0^{(\gamma_1)}
    \sbra*{\frac{h^{(\gamma_2)}(X_s)}{h^{(\gamma_1)}(X_s)}F_s},
  \end{align}
  which implies that \((\frac{h^{(\gamma_2)}(X_t)}{h^{(\gamma_1)}(X_t)}, t> 0)\) is
  a non-negative \(\bP_0^{(\gamma_1)}\)-martingale with mean \(1\).
\end{proof}

\begin{proof}[Proof of Theorem~\ref{Thm:p-gamma-abs-infty}]
  We first prove Theorem~\ref{Thm:p-gamma-abs-infty}
  in the case \(x\in \cH^{(\gamma)}\).
  Let \(F_s\) be a
  non-negative bounded \(\cF_s\)-measurable functional. Then,
  since \((\frac{1}{h^{(\gamma)}(X_t)},t\ge 0)\) is a non-negative
  \(\bP_x^{(\gamma)}\)-martingale
  by Lemma~\ref{Lem:P-g-mart}, we see
  \( \lim_{t\to\infty}\frac{1}{h^{(\gamma)}(X_t)}\)
  exists and its limit is non-negative \(\bP_x^{(\gamma)}\)-a.s.
  By Fatou's lemma, we obtain
  \begin{align}
    \bP_x^{(\gamma)}\sbra*{ \lim_{t\to\infty}\frac{1}{h^{(\gamma)}(X_t)}}
    \le \liminf_{t\to\infty}\bP_x^{(\gamma)}\sbra*{\frac{1}{h^{(\gamma)}(X_t)}}
    =\frac{1}{h^{(\gamma)}(x)}\liminf_{t\to\infty}
    \bP_x\rbra{T_0>t}
    = 0,
  \end{align}
  here the last equality follows from the fact that \((X,\bP_x)\) is recurrent.
  Hence it holds that
  \(\lim_{t\to\infty}\frac{1}{h^{(\gamma)}(X_t)} = 0\), \(\bP_x^{(\gamma)}\)-a.s.
  This implies \(\lim_{t\to\infty}\abs{X_t} = \infty\), \(\bP_x^{(\gamma)}\)-a.s.
  The proof in the case \(x=0\) is similar. Hence we omit it.
\end{proof}

\begin{proof}[Proof of Theorem~\ref{Thm:P-g-equi}]
  We first consider the case \(x\in \cH^{(\gamma)}\).
  Let \(F_t\) be
  a non-negative bounded \(\cF_t\)-measurable
  functional. Then we have
  \begin{align}
    \bP_x^{(\gamma)}\sbra{F_t}
    = \bP_x\sbra*{\frac{h^{(\gamma)}(X_t)}{h^{(\gamma)}(x)}F_t; T_0>t}
    = \frac{h(x)}{h^{(\gamma)}(x)}
    \bP_x^{(0)}\sbra*{\frac{h^{(\gamma)}(X_t)}{h(X_t)}F_t}.
  \end{align}
  Since \(h(x)\ge \abs{x}/m^2\) and
  \(h^{(\gamma)}(x) = h(x) + \gamma x / m^2\), it holds that
  \begin{align}
    1- \abs{\gamma}
    \le \frac{h^{(\gamma)}(X_t)}{h(X_t)}
    \le 1 + \abs{\gamma}.
  \end{align}
  Thus, we have
  \begin{align}
    (1-\abs{\gamma}) \frac{h(x)}{h^{(\gamma)}(x)} \bP_x^{(0)}\sbra{F_t}
    \le \bP_x^{(\gamma)}\sbra{F_t}
    \le (1+\abs{\gamma}) \frac{h(x)}{h^{(\gamma)}(x)}\bP_x^{(0)}\sbra{F_t}.
  \end{align}
  By the extension theorem, it holds that
  \begin{align}
    (1-\abs{\gamma}) \frac{h(x)}{h^{(\gamma)}(x)} \bP_x^{(0)}
    \le \bP_x^{(\gamma)}
    \le (1+\abs{\gamma}) \frac{h(x)}{h^{(\gamma)}(x)}\bP_x^{(0)},
    \quad \text{on \(\cF_\infty\).}
  \end{align}
  By the similar discussion, we also have
  \begin{align}
    (1-\abs{\gamma}) \bP_0^{(0)}
    \le \bP_0^{(\gamma)}
    \le (1+\abs{\gamma})\bP_0^{(0)},
    \quad \text{on \(\cF_\infty\).}
  \end{align}
  Therefore we obtain the desired result.
\end{proof}

\begin{proof}[Proof of Theorem~\ref{Thm:h-gamma-process-infty}]
  We first consider the case \(x\in \cH^{(\gamma)}\).
  Let \(\gamma_1,\gamma_2 \in [-1, 1]\) be
  different constants.
  Since \((\frac{h^{(\gamma_2)}(X_t)}{h^{(\gamma_1)}(X_t)}, t> 0)\) is
  a non-negative \(\bP_x^{(\gamma_1)}\)-martingale by Lemma~\ref{Lem:P-g-mart},
  the limit \(\lim_{t\to\infty}\frac{h^{(\gamma_2)}(X_t)}{h^{(\gamma_1)}(X_t)}\)
  exists and is finite
  \(\bP_0^{(\gamma_1)}\)-a.s.
  By
  (\ref{Lem-item:h/x-infty}) of Lemma~\ref{Lem:h},
  we see
  \begin{align}\label{eq:lim-hg'/hg}
    \lim_{x\to\pm\infty} \frac{h^{(\gamma_2)}(x)}{h^{(\gamma_1)}(x)}
    = \frac{1\pm \gamma_2}{1\pm \gamma_1}
    \in [0, \infty].
  \end{align}
  Hence the limits
  \(\lim_{x\to\infty} \frac{h^{(\gamma_2)}(x)}{h^{(\gamma_1)}(x)}\)
  and \(\lim_{x\to -\infty} \frac{h^{(\gamma_2)}(x)}{h^{(\gamma_1)}(x)}\)
  are different.
  Combining this and~\eqref{eq:Omega-union}, we have
  \(\bP^{(\gamma_1)}_x(\Omega^{+}_\infty\cup \Omega^{-}_\infty)=1\).
  Since \(\lim_{t\to\infty}\frac{h^{(\gamma_2)}(X_t)}{h^{(\gamma_1)}(X_t)}\)
  is finite \(\bP_x^{(\gamma_1)}\)-a.s.,
  and \( \lim_{x\to\pm\infty} \frac{h^{(\gamma_2)}(x)}{h^{(\mp 1)}(x)}
  = \infty\) for \(\gamma_2 \ne \mp 1\), we have
  \(\bP^{(1)}_x(\Omega^{+}_\infty)=\bP^{(-1)}_x(\Omega^{-}_\infty)=1\).
  Suppose \(-1<\gamma_1<1\).
  Then, since
  \(\frac{h^{(\gamma_2)}(X_t)}{h^{(\gamma_1)}(X_t)} \le \frac{1+\abs{\gamma_2}}{1-\abs{\gamma_1}}\),
  we may apply the dominated convergence theorem to obtain
  \begin{align}\label{eq:lim-inf-mart}
    \bP_x^{(\gamma_1)}\sbra*{\lim_{t\to\infty}\frac{h^{(\gamma_2)}(X_t)}{h^{(\gamma_1)}(X_t)}}
    = \lim_{t\to\infty} \bP_x^{(\gamma_1)}\sbra*{\frac{h^{(\gamma_2)}(X_t)}{h^{(\gamma_1)}(X_t)}}
    = \frac{h^{(\gamma_2)}(x)}{h^{(\gamma_1)}(x)}.
  \end{align}
  By~\eqref{eq:lim-hg'/hg} and~\eqref{eq:lim-inf-mart},
  we have
  \begin{align}\label{eq:p-g1-p-g2}
    \frac{1+\gamma_2}{1+\gamma_1}\bP_x^{(\gamma_1)}(\Omega^{+}_\infty)
    + \frac{1-\gamma_2}{1-\gamma_1}\bP_x^{(\gamma_1)}(\Omega^{-}_\infty)
    = \frac{h^{(\gamma_2)}(x)}{h^{(\gamma_1)}(x)}.
  \end{align}
  Since \(\bP_x^{(\gamma_1)}(\Omega^{+}_\infty)+
  \bP_x^{(\gamma_1)}(\Omega^{-}_\infty)=1\),~\eqref{eq:p-g1-p-g2} implies that
  \begin{align}
    \bP_x^{(\gamma_1)}(\Omega^{+}_\infty) = \frac{1+\gamma_1}{2}\frac{h^{(1)}(x)}{h^{(\gamma_1)}(x)}
    \quad\text{and}\quad
    \bP_x^{(\gamma_1)}(\Omega^{-}_\infty) = \frac{1-\gamma_1}{2}\frac{h^{(-1)}(x)}{h^{(\gamma_1)}(x)}.
  \end{align}
  The proof in the case \(x=0\) is similar. So we omit it.
\end{proof}

\section{The short-time behaviors}\label{Sec:short-time}

First, we offer the proof of Theorem~\ref{Thm:h^S/x}.

\begin{proof}[Proof of Theorem~\ref{Thm:h^S/x}]
  We write \(h^S(x)=h(x)+h(-x)\).
  Since we have
  \begin{align}
    \Re \varPsi(\lambda) = \frac{1}{2} \sigma^2 \lambda^2
    + \int_\bR \rbra*{1 - \cos \lambda x} \nu(\cd x)\ge 0,\label{eq:theta}
  \end{align}
  it holds that \(\Re(\frac{1}{\varPsi(\lambda)})\ge 0\).
  In addition, it holds that
  \(\lim_{x\to 0+}x^2 \varPsi\rbra{\lambda/x} =\sigma^2\lambda^2/2\).
  By (\ref{Lem-item:h^S-repre}) of Lemma~\ref{Lem:h}, we have
  \begin{align}
    \frac{h^S(x)}{x}
     & = \frac{2}{\pi x} \int_0^\infty
    \Re\rbra*{\frac{1-\cos\lambda x}{\varPsi(\lambda)}}\, \cd \lambda
    =  \frac{2}{\pi} \int_0^\infty
    \Re\rbra*{\frac{1-\cos\xi}{x^2\varPsi(\xi/x)}}\, \cd\xi.
  \end{align}
  We first assume \(\sigma^2=0\). By Fatou's lemma, we obtain
  \begin{align}
    \liminf_{x\to 0+}\frac{h^S(x)}{x}
     & = \liminf_{x\to 0+} \frac{2}{\pi} \int_0^\infty
    \Re\rbra*{\frac{1-\cos\xi}{x^2\varPsi(\xi/x)}}\, \cd\xi \\
     & \ge \frac{2}{\pi} \int_0^\infty\liminf_{x\to 0+}
    \Re\rbra*{\frac{1-\cos\xi}{x^2\varPsi(\xi/x)}}\, \cd\xi \\
     & =\infty,
  \end{align}
  which implies~\eqref{eq:h^S'}.

  We next assume \(\sigma^2>0\).
  Since \(\abs{x^2\varPsi(\xi/x)}\ge \abs{\Re (x^2\varPsi(\xi/x))}\ge \sigma^2\xi^2/2\),
  we have
  \begin{align}
    \abs*{\Re\rbra*{\frac{1-\cos\xi}{x^2\varPsi(\xi/x)}}}
    \le \abs*{\frac{1-\cos\xi}{x^2\varPsi(\xi/x)}}
    \le \frac{2(1\wedge\xi^2)}{ \sigma^2\xi^2},
  \end{align}
  which is integrable in \(\xi>0\).
  Hence we may apply the dominated convergence theorem to obtain
  \begin{align}
    \lim_{x\to 0+}\frac{h^S(x)}{x}
    =\frac{4}{\pi\sigma^2} \int_0^\infty
    \frac{1-\cos\xi}{\xi^2}\, \cd\xi=\frac{2}{\sigma^2},
  \end{align}
  which implies~\eqref{eq:h^S'}.
\end{proof}

\begin{proof}[Proof of Theorem~\ref{Thm:gaussian-entrance}]
  Let \(\gamma_1,\gamma_2 \in [-1, 1]\)
  be different constants.
  Then, since \(\rbra{\frac{h^{(\gamma_2)}(X_t)}{h^{(\gamma_1)}(X_t)}, t> 0}\)
  is a non-negative \(\bP_0^{(\gamma_1)}\)-martingale
  by Lemma~\ref{Lem:P-g-mart}, the limit
  \(\lim_{t\to 0+}\frac{h^{(\gamma_2)}(X_t)}{h^{(\gamma_1)}(X_t)}\) exists
  \(\bP_0^{(\gamma_1)}\)-a.s.\ for all \(t>0\).
  We have
  \begin{align}
    \lim_{x\to 0\pm}\frac{h^{(\gamma_2)}(x)}{h^{(\gamma_1)}(x)}
    = \frac{\abs{h^{(\gamma_2)\prime}(0\pm)}}{\abs{h^{(\gamma_1)\prime}(0\pm)}}
    =\frac{\abs{h'(0\pm)}\pm \gamma_2/m^2}{\abs{h'(0\pm)}\pm \gamma_1/m^2}
    \in [0,\infty].
  \end{align}
  Consequently, the limits
  \( \lim_{x\to 0+}\frac{h^{(\gamma_2)}(x)}{h^{(\gamma_1)}(x)}\)
  and \( \lim_{x\to 0-}\frac{h^{(\gamma_2)}(x)}{h^{(\gamma_1)}(x)}\)
  are different,
  which yields that \(\bP_0^{(\gamma)}(\Omega_0^{+,-})=0\).
  Since the \(\bP_0^{(\gamma_1)}\)-martingale
  \(\rbra{\frac{h^{(\gamma_2)}(X_t)}{h^{(\gamma_1)}(X_t)}, t>0}\)
  has mean \(1\),
  it holds that
  \begin{align}\label{eq:zero-mean1}
    \frac{{h^{(\gamma_2)\prime}}(0+)}{{h^{(\gamma_1)\prime}}(0+)}
    \bP_0^{(\gamma_1)}(\Omega_0^+)
    + \frac{\abs{h^{(\gamma_2)\prime}(0-)}}{\abs{h^{(\gamma_1)\prime}(0-)}}
    \bP_0^{(\gamma_1)}(\Omega_0^-)=1.
  \end{align}
  Since \(\bP_0^{(\gamma_1)}(\Omega_0^+)+ \bP_0^{(\gamma_1)}(\Omega_0^-)=1\)
  and by Theorem~\ref{Thm:h^S/x}, we obtain
  \begin{align}
    \bP_0^{(\gamma_1)}(\Omega_0^+)
    = \frac{\sigma^2}{2}{h^{(\gamma_1)\prime}}(0+)
    \quad\text{and}\quad
    \bP_0^{(\gamma_1)}(\Omega_0^-)=\frac{\sigma^2}{2}\abs{{h^{(\gamma_1)\prime}}(0-)}.
  \end{align}
  Hence the proof is complete.
\end{proof}

\begin{proof}[Proof of Theorem~\ref{Thm:entrance-one}]
  Assume \(h'(0+)=\infty\) and \(\abs{h'(0-)}<\infty\).
  Let \(\gamma_1,\gamma_2 \in [-1, 1]\)
  be different constants.
  By the same discussion as the proof of Theorem~\ref{Thm:gaussian-entrance},
  we obtain~\eqref{eq:zero-mean1}. By the assumption, we also have
  \(\frac{h^{(\gamma_2)\prime}(0+)}{h^{(\gamma_1)\prime}(0+)}=1\)
  for all \(-1\le\gamma_2\le 1\). Thus we have
  \begin{align}
    \bP_0^{(\gamma_1)}(\Omega_0^+)
    = 1
    \quad\text{and}\quad
    \bP_0^{(\gamma_1)}(\Omega_0^-)=0.
  \end{align}
  The proof in the case \(h'(0+)<\infty\) and \(\abs{h'(0-)}=\infty\)
  is similar. So we omit it.
\end{proof}

\begin{proof}[Proof of Theorem~\ref{Thm:feller}]
  Ikeda--Watanabe~\cite[Theorem 3.3]{MR0451425} proved that
  \begin{align}
    \bP_x\rbra{\Omega_1^{+,-}| T_0<\infty} = 1,
    \quad x\in \bR\setminus\cbra{0},
  \end{align}
  where
  \begin{align}
    \Omega_1^{+,-}\coloneqq
    \cbra*{\text{\(\exists \cbra{t_n}\)
    with \(t_n\to {T_0-}\) such that \(\forall n,\, X_{t_n}X_{t_{n+1}}<0\)}}.
  \end{align}
  This implies that
  \begin{align}
    n((\Omega_1^{+,-})^c\cap \cbra{T_0<\infty})=0.
  \end{align}
  By time reversal property of excursion paths (see~\cite[Lemma 5.2]{MR2397787}),
  it holds that
  \begin{align}
    n((\Omega_0^{+,-})^c\cap \cbra{T_0<\infty})=0.
  \end{align}
  Since \((X,\bP)\) is recurrent,
  it holds that
  \(n(\cbra{T_0<\infty}^c)=0\). Thus we have \(n((\Omega_0^{+,-})^c)=0\),
  which implies
  \( \bP_0^{(\gamma)}(\Omega_0^{+,-})=1\).
\end{proof}
\section{Appendix A: Resolvent density under \(\bP_x^{(\gamma)}\)}\label{Sec:resolvent}
We calculate the resolvent density under \(\bP_x^{(\gamma)}\)
and show some Feller property.
Recall that we always assume the assumption~\ref{item:assumption}.

Let \(p_t(\cd x)\) denote the transition law of \(X_t\) under \(\bP\)
and let \(p_t^0(x, \cd y)\) denote
the transition law of \(X_t\) under \(\bP_x^0\).
By the Markov property,
we have, for \(x, y\in \bR\setminus\cbra{0}\),
\begin{align}\label{eq:p_t^0}
  p_t^0(x, \cd y)= p_t(\cd y-x) -\int_{[0,t]} \bP_x(T_0\in \cd s)p_{t-s}(\cd y).
\end{align}
For \(t,q>0\) and \(x,y \in \bR\setminus\cbra{0}\),
we denote the \(q\)-resolvent for killed process by
\begin{align}
  r_q^0(x,y) & = \int_0^\infty \ce^{-qt}p_t^0(x,\cd y)\,\cd t/\cd y        \\
             & = r_q(y-x)-\frac{r_q(-x)r_q(y)}{r_q(0)}                     \\
             & = h_q(x)+h_q(-y) - h_q(x-y) - \frac{h_q(x)h_q(-y)}{r_q(0)}.
\end{align}
Note that the second identity follows from~\eqref{eq:p_t^0}
and~\eqref{eq:-qT_0}.
This implies that the killed process
\((X,\bP_x^0)\) has the continuous \(q\)-resolvent density.

Let \(-1\le \gamma \le 1\).
For \(x\in \cH^{(\gamma)}_0\) and
\(y\in \cH^{(\gamma)}\), we denote the transition law of \(X_t\) under
\(\bP^{(\gamma)}_x\)
by
\begin{align}
  p_t^{(\gamma)}(x,\cd y) =
  \begin{dcases}
    \frac{h^{(\gamma)}(y)}{h^{(\gamma)}(x)}
    p_t^0(x,\cd y)                 & x\in\cH^{(\gamma)}, \\
    h^{(\gamma)}(y) n(X_t\in\cd y) & x = 0.
  \end{dcases}
\end{align}
Then the \(q\)-resolvent density
\(r_q^{(\gamma)}(x,y)\) of \((X,\bP_x^{(\gamma)})\) can be expressed as follows:
if \(x\in \cH^{(\gamma)}\),
\begin{align}
  r_q^{(\gamma)}(x,y)
   & = \int_0^\infty \ce^{-qt} p_t^{(\gamma)}(x,\cd y) \, \cd t /\cd y \\
   & = \frac{h^{(\gamma)}(y)}{h^{(\gamma)}(x)}
  \rbra*{h_q(x)+h_q(-y)-h_q(x-y)-\frac{h_q(x)h_q(-y)}{r_q(0)}},
\end{align}
and
\begin{align}
  r_q^{(\gamma)}(0, y)
   & = \int_0^\infty \ce^{-qt} p_t^{(\gamma)}(0,\cd y) \, \cd t /\cd y
  = \frac{h^{(\gamma)}(y)r_q(y)}{r_q(0)}
  =  h^{(\gamma)}(y)\rbra*{1-\frac{h_q(-y)}{r_q(0)}},
\end{align}
where the second identity follows from the formula:
for any non-negative measurable function \(f\), it holds that
\begin{align}\label{eq:formula-duality}
  \int_0^\infty \ce^{-qt} n\sbra{f(X_t)}\, \cd t
  = \int_\bR f(x) \widehat{\bP}_x\sbra{\ce^{-qT_0}}\, \cd x,
\end{align}
where \(\bP_x\) and \(\widehat{\bP}_x\) are in weak duality, i.e.,
the probability measure \(\widehat{\bP}_x\) denotes the law of \((-X_t, t\ge 0)\)
under \(\bP_{-x}\).
For more details, see Chen--Fukushima--Ying~\cite{MR2397787}
and Fitzsimmons--Getoor~\cite{MR2247835}.
See also Yano--Yano--Yor~\cite[Theorem 3.3]{MR2599211}.

Summarizing the above computations, we obtain the following results:
\begin{Prop}\label{Prop:gamma-q-resolvent}
  Let \(q>0\) and \(-1\le\gamma\le 1\). Let \(r_q^{(\gamma)}(x,y)\) denote the \(q\)-resolvent
  density of \((X,\bP_x^{(\gamma)})\). Then,
  for \(y\in \cH^{(\gamma)}\), it holds that
  \begin{align}\label{eq:r_q^g}
    r_q^{(\gamma)}(x,y)
    =\begin{dcases}
      \frac{h^{(\gamma)}(y)}{h^{(\gamma)}(x)}
      \rbra*{h_q(x)+h_q(-y)-h_q(x-y)-\frac{h_q(x)h_q(-y)}{r_q(0)}}
       & \text{if \(x\in \cH^{(\gamma)}\)}, \\
      h^{(\gamma)}(y)\rbra*{1-\frac{h_q(-y)}{r_q(0)}}
       & \text{if \(x=0\).}
    \end{dcases}
  \end{align}
\end{Prop}

Letting \(q\to 0+\) in~\eqref{eq:r_q^g},
we obtain the zero resolvent
\begin{align}
  r_0^{(\gamma)}(x,y)\coloneqq \lim_{q\to 0+} r_q^{(\gamma)}(x,y).
\end{align}
Note that it holds that \(\lim_{q\to 0+}\frac{1}{r_q(0)}=0\)
if \((X,\bP)\) is recurrent;
see, e.g.,~\cite[Theorem I.17]{MR1406564}
and~\cite[Theorem 37.5]{MR1739520}.
(Recall that we always assume \((X,\bP)\) is recurrent.)
\begin{Cor}
  For \(y\in \cH^{(\gamma)}\), it holds that
  \begin{align}
    r_0^{(\gamma)}(x,y)
    =\begin{dcases}
      \frac{h^{(\gamma)}(y)}{h^{(\gamma)}(x)}
      \rbra*{h(x)+h(-y)-h(x-y)}
                      & \text{if \(x\in \cH^{(\gamma)}\),} \\
      h^{(\gamma)}(y) & \text{if \(x=0\).}
    \end{dcases}
  \end{align}
\end{Cor}

Set
\begin{align}
  T_t^{(\gamma)}f(x)= \bP_x^{(\gamma)}\sbra{f(X_t)},
  \quad t\ge 0,\; f\in \cB_{+, b}(\cH_0^{(\gamma)}),
\end{align}
where \(\cB_{+, b}(\cH_0^{(\gamma)})\) denotes
the set of non-negative bounded measurable functions.
Then the family \(\cbra{T_t^{(\gamma)},t\ge 0}\)
forms a transition semigroup.
We define the resolvent operator of the semigroup \(T_t^{(\gamma)}\) as
\begin{align}
  R_q^{(\gamma)}f(x) = \int_0^\infty \ce^{-qt}T_t^{(\gamma)}f(x)\, \cd t,\quad
  q>0,\; f\in \cB_{+, b}(\cH_0^{(\gamma)}).
\end{align}
For Theorem~\ref{Thm:fellerp}, we are inspired by
Yano~\cite[Theorem 1.5]{MR2603019}.
\begin{Thm}\label{Thm:fellerp}
  Assume the conditions (\ref{item:cond:h-gamma-infty})--(\ref{item:cond:h-/h})
  of Theorem~\ref{Thm:feller} hold.
  (Note that~(\ref{item:cond:h-gamma-infty}) of Theorem~\ref{Thm:feller}
  and subadditivity of \(h^{(\gamma)}\)
  imply that \(\cH_0^{(\gamma)}=\bR\).)
  Then the semigroup \(\rbra{T_t^{(\gamma)}}_{t\ge 0}\) enjoys Feller property, i.e.,
  \begin{enumerate}[label=\textbf{(F\arabic*)},series=feller]
    \item \(T^{(\gamma)}_tC_0(\bR)\subset C_0(\bR)\);\label{item:feller1}
    \item \(\norm{T_t^{(\gamma)}f - f}\to 0\) as \(t\to 0+\) for
          all \(f\in C_0(\bR)\),\label{item:feller2}
  \end{enumerate}
\end{Thm}
where \(C_0(\bR)\) stands for the class of
continuous functions vanishing at infinity.
\begin{proof}
  Note that the condition~(\ref{item:cond:h/x-infty}) implies that
  \begin{align}
    \lim_{x\to 0}\frac{h^{(\gamma)}(x)}{\abs{x}}= \infty
    \quad\text{and}\quad
    \lim_{x\to 0}\frac{h^{(\gamma)}(x)}{h(x)}= 1,
  \end{align}
  for \(-1\le\gamma \le 1\).
  To show Feller property,
  it is sufficient to show that
  \begin{enumerate}[resume*=feller]
    \item \(T_t^{(\gamma)}f(x)\to f(x)\) as \(t\to 0+\) for all
          \(x\in\bR,\; f\in C_0(\bR)\),\label{item:feller3}
    \item \(R_q^{(\gamma)}C_0(\bR)\subset C_0(\bR)\).\label{item:feller4}
  \end{enumerate}
  For more details see~\cite[Proposition III.2.4]{RevuzYor}.
  Since \(T_t^{(\gamma)}f(x)=\bP_x^{(\gamma)}\sbra{f(X_t)}\)
  and since \(\bP_x^{(\gamma)}\) is a probability measure on the
  c\`{a}dl\`{a}g space,~\ref{item:feller3} is obvious.
  We proceed the proof of~\ref{item:feller4}.
  Let \(C_c(\bR)\) stand for the set of continuous functions with compact support
  on \(\bR\).
  Since \(\norm{qR_q^{(\gamma)}f}\le
  \norm{f}\)
  and since the closure of \({C_c(\bR)}\) is \(C_0(\bR) \),
  it is sufficient to show that \(R_q^{(\gamma)}C_c(\bR)\subset C_0(\bR)\).
  Recall that, for \(x\in\bR\),
  \begin{align}
    R_q^{(\gamma)}f(x)
    = \int_{\cH^{(\gamma)}} f(y)r_q^{(\gamma)}(x,y)\, \cd y,
    \quad f\in C_c(\bR).
  \end{align}
  Since \(f\) has compact support and is
  continuous, and \(r_q^{(\gamma)}\) is continuous in
  \((x,y)\in \cH^{(\gamma)}\times \cH^{(\gamma)}\),
  the function \(R_q^{(\gamma)}f(x)\) is continuous
  in \(x\in\cH^{(\gamma)}\).

  Let the set \(A\subset\bR\) stand for the support of \(f\).
  Since \(r_q^{(\gamma)}(x,y)=\frac{h^{(\gamma)}(y)}{h^{(\gamma)}(x)}r_q^0(x,y)\),
  it holds that
  \begin{align}
    R_q^{(\gamma)}f(x)
     & = \frac{1}{h^{(\gamma)}(x)}\int_{\cH^{(\gamma)}} f(y)h^{(\gamma)}(y) r_q^0(x,y)
    \, \cd y                                                                           \\
     & \le \sup_{y\in A} h^{(\gamma)}(y) \frac{1}{q
    h^{(\gamma)}(x)} \int_{\cH^{(\gamma)}}  f(y) r_q^0(x,y)\, \cd y                    \\
     & \le \sup_{y\in A} h^{(\gamma)}(y) \frac{\norm{f}}{q
    h^{(\gamma)}(x)}\bP_x\sbra{\ce^{-qT_A}}                                            \\
     & \to 0 \quad \text{as \(x\to\pm\infty\).}
  \end{align}
  Here we used the assumption~(\ref{item:cond:h-gamma-infty}).
  Hence \(R_q^{(\gamma)}f(x)\) vanishes at infinity.

  We have to prove \(R_q^{(\gamma)}f(x)\) is continuous at \(x=0\).
  By Proposition~\ref{Prop:gamma-q-resolvent}, and
  assumptions~(\ref{item:cond:h/x-infty})
  and~(\ref{item:cond:h_q/h}),
  we have \(r_q(x,y)\to r_q(0,y)\) as \(x\to 0\).
  Moreover, since \(h_q\) is subadditive,
  it holds that, for \(x,y\in \cH^{(\gamma)}\),
  \begin{align}
    r_q^{(\gamma)}(x,y)\le h^{(\gamma)}(y)
    \frac{h_q(x)+h_q(-x)}{h^{(\gamma)}(x)}.
    \label{eq:r_q-bdd}
  \end{align}
  The conditions~(\ref{item:cond:h_q/h}) and~(\ref{item:cond:h-/h}) implies that
  the right hand side of~\eqref{eq:r_q-bdd} is bounded near \(x=0\)
  and \(y\in A\).
  Thus we may apply the dominated convergence theorem to deduce that
  \(R_q^{(\gamma)}f(x)\) is also continuous at \(x=0\).
  Hence \((T_t^{(\gamma)})\) has Feller property.
\end{proof}
\section{Appendix B: Proof of Theorem~\ref{Thm:limit-P0}}\label{Sec:pf-meander}
Recall that we always assume the assumption~\ref{item:assumption}.
\begin{Lem}\label{Lem:integ-h-}
  For any \(t>0\), it holds that \(n\sbra{h(-X_t)}<\infty\).
\end{Lem}
Recall that \(n\sbra{h(X_t)}=1\) for all \(t>0\); see
Lemma~\ref{Lem:harmonic}.
\begin{proof}[Proof of Lemma~\ref{Lem:integ-h-}]
  We write \(\hat{h}(x)=h(-x)\).
  By the formula~\eqref{eq:formula-duality}
  and by~\eqref{eq:-qT_0},
  we have
  \begin{align}
    \int_0^\infty \ce^{-qt} n\sbra{\hat{h}(X_t)}\,\cd t
     & = \int_\bR \hat{h}(x)\widehat{\bP}_x\sbra{\ce^{-qT_0}}\,\cd x           \\
     & = \int_\bR \hat{h}(x)\frac{r_q(x)}{r_q(0)}\,\cd x.\label{eq:h-hat-inte}
  \end{align}
  By~\cite[(3.20)]{MR3689384}, the equation~\eqref{eq:h-hat-inte} is finite.
  (Note that the assumptions in~\cite{MR3689384} is stronger,
  but this remains true since its proof is valid if Lemma~\ref{Lem:h} holds.)
  Hence, for almost any \(t>0\),
  it holds that \(n\sbra{\hat{h}(X_t)}<\infty\).
  Thus for any \(t>0\), there exists \(0<s<t\) such that
  \(n\sbra{\hat{h}(X_s)}<\infty\).
  By the Markov property of the excursion measure \(n\), we have
  \begin{align}
    n\sbra{\hat{h}(X_t)}
     & = n\sbra{\bP_{X_s}\sbra{\hat{h}(X_{t-s}); T_0>t-s}}                   \\
     & \le  n\sbra{\bP_{X_s}\sbra{\hat{h}(X_{t-s})}}                         \\
     & = n\sbra{\widetilde{\bP}_{0}\sbra{\hat{h}(X_s+\widetilde{X}_{t-s})}},
  \end{align}
  where the symbol \(\widetilde{\hspace{12pt}}\) means independence.
  Since \(\hat{h}\) is subadditive,
  we obtain
  \begin{align}
    n\sbra{\hat{h}(X_t)}\le
    n\sbra{\hat{h}(X_s)}+\bP_{0}\sbra{\hat{h}(X_{t-s})}.
  \end{align}
  By Tsukada~\cite[Proof of Theorem 15.2]{MR3838874}
  (see also Takeda--Yano~\cite[Lemma 4.3]{me}),
  we have \(\bP_{0}\sbra{\hat{h}(X_{t-s})}<\infty\).
  Consequently, it holds that
  \(n\sbra{\hat{h}(X_t)}<\infty\) for all \(t>0\).
\end{proof}

For the proof of Theorem~\ref{Thm:limit-P0},
we introduce the following lemma, whose proof
is in~\cite[Lemmas 3.4 and 6.2]{me}.
\begin{Lem}[\cite{me}]\label{Lem:hitting-h-repre}
  For \(a,b\in \bR\setminus\cbra{0}\) and \(a\ne b\), it holds that
  \begin{align}
     & h^B(a)\coloneqq  \bP_0\sbra{L_{T_a}}
    = h(a)+h(-a),                                                \\
     & h^B(a)\bP_x(T_a<T_0)
    = h(x)+h(-a)-h(x-a),                                         \\
     & \bP_0\sbra{L_{T_{\cbra{a,-b}}}}\bP_x(T_{\cbra{a,-b}}<T_0) \\
     & =h(x) + \frac{1}{h^B(a+b)}
    \cbra*{\begin{multlined}
        \rbra[\big]{h(-a)-h(x-a)}h(a+b)
        + \rbra[\big]{h(b)-h(x+b)}h(-a-b)\\
        -\rbra[\big]{h(a)-h(-b)}\rbra[\big]{h(-a)-h(x-a)-h(b)+h(x+b)}
      \end{multlined}}.
  \end{align}
\end{Lem}

\begin{proof}[Proof of Theorem~\ref{Thm:limit-P0}]
  For \(s>0\),
  we define \(d_s=\inf\cbra{u>s\colon X_u=0}\).
  We also define \(G=\cbra{g_s\colon g_s\ne d_s, s>0}\).

  \noindent (\ref{Thm-item:limit-P0-exp})
  For any \(q>0\), we have
  \begin{align}
    \bP_0\sbra{F_t\circ k_{\bm{e}_q-g_{\bm{e}_q}}\circ \theta_{g_{\bm{e}_q}}}
     & =\bP_0\sbra*{\int_0^\infty q\ce^{-qu}F_t\circ k_{u-{g_u}}\circ \theta_{g_u}\,
    \cd u}                                                                           \\
     & = \bP_0\sbra*{\sum_{s\in G}\ce^{-qs} \int_s^{d_s} q\ce^{-q(u-s)}F_t
      \circ k_{u-s}\circ \theta_{s}\, \cd u}.
  \end{align}
  Using the compensation formula in excursion theory (see e.g.,
  Bertoin~\cite[Corollary IV.11]{MR1406564}), we obtain
  \begin{align}
    \bP_0\sbra{F_t\circ k_{\bm{e}_q-g_{\bm{e}_q}}\circ \theta_{g_{\bm{e}_q}}}
    =\bP_0 \sbra*{\int_0^\infty \ce^{-qs}\, \cd L_s}
    n\sbra*{\int_0^{T_0} q\ce^{-qu} F_t 1_{\cbra{u>t}}\, \cd u}.
  \end{align}
  By~\eqref{eq:regularity-of-L} and the Markov property of the excursion measure \(n\),
  it holds that
  \begin{align}
    \bP_0\sbra{F_t\circ k_{\bm{e}_q-g_{\bm{e}_q}}\circ \theta_{g_{\bm{e}_q}}}
     & = r_q(0) n\sbra{F_t; t<\bm{e}_q<T_0}                       \\
     & = r_q(0)\ce^{-qt}n\sbra{F_t \bP_{X_t}\sbra{T_0>\bm{e}_q}}.
  \end{align}
  It follows from~\eqref{eq:-qT_0} that
  \begin{align}
    \bP_{X_t}\sbra{T_0>\bm{e}_q}=1-\bP_{X_t}\sbra{\ce^{-qT_0}}
    = 1-\frac{r_q(-X_t)}{r_q(0)}=\frac{h_q(X_t)}{r_q(0)}.
  \end{align}
  Hence we have
  \begin{align}
    \bP_0\sbra{F_t\circ k_{\bm{e}_q-g_{\bm{e}_q}}\circ \theta_{g_{\bm{e}_q}}}
    = \ce^{-qt}n\sbra{F_t h_q(X_t)}.
  \end{align}
  By~\eqref{eq:h_q},~\eqref{eq:theta} and (\ref{Lem-item:h^S-repre}) of
  Lemma~\ref{Lem:h},
  it holds that
  \begin{align}
    h_q(X_t) & \le h_q(X_t)+h_q(-X_t)
    =\frac{2}{\pi}\int_0^\infty
    \Re\rbra*{\frac{1-\cos \lambda x}{q+\varPsi(\lambda)}} \, \cd \lambda \\
             & \le \frac{2}{\pi}\int_0^\infty
    \Re\rbra*{\frac{1-\cos \lambda x}{\varPsi(\lambda)}} \, \cd \lambda
    = h(X_t)+h(-X_t).
  \end{align}
  By Lemmas~\ref{Lem:harmonic} and~\ref{Lem:integ-h-},
  the function
  \(h(X_t)+h(-X_t)\) is integrable with respect to the measure \(n\).
  Thus we may apply the dominated convergence theorem
  to deduce
  \begin{align}
    \lim_{q\to 0+}\bP_0\sbra{F_t\circ k_{\bm{e}_q-g_{\bm{e}_q}}\circ \theta_{g_{\bm{e}_q}}}
    = \lim_{q\to 0+}\ce^{-qt}n\sbra{F_t h_q(X_t)}
    =n\sbra{F_t h(X_t)}
    = \bP_0^{(0)}\sbra{F_t}.
  \end{align}

  \noindent (\ref{Thm-item:limit-P0-hitting})
  For \(a\in \bR\setminus\cbra{0}\), we have
  \begin{align}
    \bP_0\sbra{F_t\circ k_{{T_a}-g_{{T_a}}}\circ \theta_{g_{{T_a}}}}
    =\bP_0\sbra*{\sum_{s\in G} 1_{\cbra{s<T_a<d_s}}
      F_t\circ k_{{T_a}-s}\circ \theta_{s}}.
  \end{align}
  Using the compensation formula in excursion theory,
  the Markov property of the excursion measure \(n\)
  and Lemma~\ref{Lem:hitting-h-repre},
  it holds that
  \begin{align}
    \bP_0\sbra{F_t\circ k_{{T_a}-g_{{T_a}}}\circ \theta_{g_{{T_a}}}}
     & =\bP_0\sbra*{\int_0^{T_a}\, \cd L_s}
    n\sbra{F_t; t<T_a<T_0}                                      \\
     & = \bP_0\sbra{L_{T_a}}n\sbra{F_t\bP_{X_t}(T_a<T_0);t<T_a} \\
     & = n\sbra{F_t ( h(-a)+h(X_t)-h(X_t-a)); t<T_a}.
  \end{align}
  Since \(h\) is subadditive,
  we have
  \begin{align}
    h(-a)+h(X_t)-h(X_t-a)\le h(X_t)+h(-X_t),
  \end{align}
  which is integrable with respect to the measure \(n\).
  Thus we may apply the dominated convergence theorem to deduce
  \begin{align}
    \lim_{a\to\pm\infty}
    \bP_0\sbra{F_t\circ k_{{T_a}-g_{{T_a}}}\circ \theta_{g_{{T_a}}}}
    = n\sbra{F_t h^{(\pm 1)}(X_t)}
    =\bP_0^{(\pm 1)}\sbra{F_t},
  \end{align}
  here we used (\ref{Lem-item:h-diff-infty}) of Lemma~\ref{Lem:h}.

  (\ref{Thm-item:limit-meas-twohitting})
  By the same discussion as the proof of (\ref{Thm-item:limit-P0-hitting}),
  it holds that
  \begin{align}
    \bP_0\sbra{F_t\circ k_{{T_{\cbra{a,-b}}}-g_{T_{\cbra{a,-b}}}}
    \circ \theta_{g_{T_{\cbra{a,-b}}}}}
    = \bP_0\sbra{L_{T_{\cbra{a,-b}}}}n\sbra{F_t\bP_{X_t}(T_{\cbra{a,-b}}<T_0)
    ;t<T_{\cbra{a,-b}}}.
  \end{align}
  By Lemma~\ref{Lem:hitting-h-repre} and by the dominated convergence theorem,
  we obtain the desired result. (We omit the details.)
\end{proof}

\noindent\textbf{Acknowledgements}
The author would like
to express his deep gratitude to Professor Kouji Yano
for his helpful advice and encouragement.
This work was supported by
JSPS Open Partnership Joint Research Projects grant no. JPJSBP120209921.

\noindent\textbf{Data Availability}
Data sharing is not applicable to this article as no datasets
were generated or analyzed during the current study.

\section*{Declarations}
\textbf{Conflict of interest}
The authors have no competing interests to declare that are relevant to the content of
this article.

\end{document}